\documentclass[12pt,reqno]{amsart}

\usepackage{amssymb,hyperref}

\numberwithin{equation}{section}

\begin{document}

\newtheorem{theorem}{Theorem}[section]
\newtheorem{proposition}[theorem]{Proposition}
\newtheorem{lemma}[theorem]{Lemma}
\newtheorem{corollary}[theorem]{Corollary}

\theoremstyle{definition}
\newtheorem{definition}[theorem]{Definition}

\newcommand\Ga{{\Gamma}}
\newcommand\Z{{\mathbb Z}}

\title[Higgs bundles on Sasakian manifolds]{Higgs bundles on Sasakian manifolds}

\author[I. Biswas]{Indranil Biswas}
\address{School of Mathematics, Tata Institute of Fundamental
Research, Homi Bhabha Road, Bombay 400005, India}
\email{indranil@math.tifr.res.in}

\author[M. Mj]{Mahan Mj}
\address{School of Mathematics, Tata Institute of Fundamental
Research, Homi Bhabha Road, Bombay 400005, India}
\email{mahan@math.tifr.res.in}



\date{}

\begin{abstract}
We extend the Donaldson-Corlette-Hitchin-Simpson
correspondence between Higgs bundles and flat connections on compact K\"ahler manifolds
to compact quasi-regular Sasakian manifolds. A particular consequence is the translation
of restrictions on K\"ahler groups proved using the Donaldson-Corlette-Hitchin-Simpson
correspondence to fundamental groups of compact Sasakian manifolds (not necessarily
quasi-regular). 
\end{abstract}

\maketitle

\tableofcontents

\section{Introduction}\label{se1}

The aim of this paper is to extend the Donaldson-Corlette-Hitchin-Simpson 
correspondence between Higgs bundles and flat connections on compact K\"ahler 
manifolds to the class of compact quasi-regular Sasakian manifolds. In order to do 
this, we are led to the natural notions of a Higgs vector bundle on a Sasakian 
manifold $M$ and that of a holomorphic Sasakian principal $G-$bundle, where $G$ is 
a connected complex reductive algebraic group.
Further, in order to establish a canonical correspondence, we are forced to look at 
a special class of Sasakian $G$--Higgs bundles that naturally descend to the base 
projective variety $M/\text{U}(1)$ of the Sasakian manifold $M$, at least up to a 
finite sheeted cover of the base. Such Higgs bundles are referred to as virtually basic 
(Definition \ref{def2}); see Section \ref{ex} for an example showing the necessity of this hypothesis.
 The following is the main result proved here (see Theorem \ref{main}).

\begin{theorem}\label{thm0}
Let $G$ be a connected reductive complex affine algebraic group. Let $M$ be a
quasi-regular Sasakian manifold
with fundamental group $\Gamma$. Any homomorphism 
$$\rho\, :\, \Gamma\,\longrightarrow\, G$$ with the Zariski closure of $\rho(\Gamma)$
reductive canonically gives a virtually basic polystable principal $G$--Higgs bundle on
$M$ with vanishing rational characteristic classes. Conversely, any 
virtually basic polystable principal $G$--Higgs bundle on $M$ with 
vanishing rational characteristic classes corresponds to a flat principal
$G$--bundle on $M$ with the property that the Zariski closure of the monodromy
representation is reductive.
\end{theorem}

Given any compact Sasakian manifold, its Riemannian
structure and the Reeb vector field can be 
perturbed to make it quasi-regular Sasakian \cite{Ru}, \cite{OV}. So the fundamental
group of any compact Sasakian manifold is actually the 
fundamental group of some compact quasi-regular Sasakian manifold. Therefore, in view
of Theorem \ref{thm0} all 
restrictions on K\"ahler groups proved using the Donaldson-Corlette-Hitchin-Simpson 
correspondence (see \cite[p.~53, Lemma 4.7]{Si2}) are applicable to the fundamental 
groups of compact Sasakian manifolds.

\section{Reductive representations of central extensions} 

We shall use below a purely group-theoretic description of the Euler class. Let
$$H \,=\, \langle g_1,\, \cdots,\, g_n\,\mid\, r_1,\, \cdots,\, r_m \rangle $$
be a finitely presented group with generators $g_i$, $1\, \leq\, i\,\leq\, n$,
and relations $r_j$, $1\, \leq\, j\,\leq\, m$, and
let $X$ be a 2-complex realizing this presentation. 
Then a circle bundle $p\,:\, E \,\longrightarrow\, X$ over $X$ corresponds to an exact
sequence
$$\Z \,\longrightarrow\, \pi_1(E)\,\longrightarrow\, \pi_1(X) \,\longrightarrow\, 1$$ and
the Euler class gives the obstruction to constructing a continuous
section of the projection $p$. The obstruction class can be evaluated on every 2-cell 
in $X$. Group theoretically, if
$$1\,\longrightarrow\, \Z \,\longrightarrow\, \pi_1(E)\,\longrightarrow\, \pi_1(X)
\,\longrightarrow\, 1$$
is a 
central extension, then the obstruction can be computed algebraically. Each 2-cell of 
$X$ corresponds to some relation $r_i$. Also, if $t$ denotes a generator for the 
central $\Z$, then we can first construct a trivialization of the circle bundle over 
the one-skeleton of $X$. Given this, the restriction of $E$ over a 2-cell $\sigma_i$ 
of $X$ induces the obstruction class. Group theoretically this corresponds to $r_i = 
t^{n_i}$ for some integer $n_i$ which coincides with the Euler class evaluated on the 
2-cell $\sigma_i$. We use this description in Lemma \ref{fin}.

\begin{lemma}\label{fin}
Let $\Ga$ be a finitely presented central extension of a group $Q$
by $\Z$ with non-zero 
Euler class. Let $G$ be a connected reductive complex affine algebraic group and
$\rho\,:\, \Ga \,\longrightarrow\, G$ a homomorphism such that
$\rho(\Ga)$ is Zariski dense in $G$. Then $\rho (\Z)$ is finite.
\end{lemma}

\begin{proof}
Let
\begin{equation}\label{ex1}
1\,\longrightarrow\, \Z\,\longrightarrow\, \Ga\,\longrightarrow\, Q\,\longrightarrow\, 1
\end{equation}
be the short exact sequence of the central extension $\Ga$.
Since $\rho (\Ga)$ is Zariski dense in $G$, and $\rho(\Z)$ commutes with
$\rho (\Ga)$, it follows that $\rho (\Z)$ lies in the center of $G$. Let $Z(G)\, \subset\,
G$ be the center. The natural homomorphism
$$
h\, :\, G\, \longrightarrow\, (G/[G,\, G])\times (G/Z(G))
$$
is surjective with finite kernel because $G$ is reductive. We note that $G/[G,\, G]$ is isomorphic to a product
of copies of the multiplicative group of nonzero complex numbers,
and $G/Z(G)$ is semi-simple with trivial center. Therefore, the homomorphism
$h\circ \rho$ sends the central $\Z\, \subset\, \Gamma$ into the subgroup
$G/[G,\, G]\times\{1\}\, \subset\, (G/[G,\, G])\times (G/Z(G))$. Consequently, $h\circ
\rho$ descends to a homomorphism
$$
h'\, :\, Q \, \longrightarrow\, G/Z(G)\, .
$$
Note that
$h\circ \rho(\Gamma)$ is Zariski dense in $(G/[G,\, G])\times (G/Z(G))$ because
$\rho(\Gamma)\, \subset\, G$ is Zariski dense and $h$ is surjective. Hence
$h'(Q)$ is Zariski dense in $G/Z(G)$. It should be clarified that
we do not exclude the possibility that the latter group is trivial.
What we shall use below is the fact that $h'(Q)$ may be regarded as a subgroup of 
$\{ 1\}\times (G/Z(G)) \subset (G/[G,\, G])\times (G/Z(G))$.

Let $\langle q_1,\, \cdots, \,q_m \,:\, r_1,\, \cdots ,\, r_s\rangle$ be a presentation
for $Q$. Then
there exists a presentation of $\Ga$ of the form 
$$\langle t,\, q_1,\, \cdots,\, q_m \,:\, [t,\, q_1],\, \cdots,\, [t,\, q_m],
\, r_1t^{i_1}, \,\cdots ,\, r_st^{i_s}\rangle\, .$$
Since the Euler class of the extension in \eqref{ex1} is non-zero, one of
the $t^{i_j}$'s is non-zero (see the discussion on Euler class preceding the lemma).
Using this and the fact that $h'(Q)$ is Zariski dense in $G/Z(G)$, it follows that 
$$h\circ \rho (t^{i_j})
\,=\, h\circ \rho (r_j^{-1})\,\in \,
((G/[G,\, G])\times\{1\})\cap (\{1\}\times (G/Z(G)))\,=\, \{ 1 \}\, ,$$
forcing $h\circ\rho (\Z)$ to be finite. This implies that $\rho (\Z)$ is finite because
the kernel of $h$ is finite.
\end{proof}

\begin{corollary}\label{torsionfree}
Take $\Gamma$ and $\rho$ as in Lemma \ref{fin}. Then the image
$\rho (\Gamma)$ is virtually torsion-free. Hence there exists a finite index subgroup
$\Gamma_1$ of $\Gamma$ such that the image $\rho (\Z \cap \Gamma_1)$ is trivial.
\end{corollary}

\begin{proof}
Since $\Gamma$ is finitely presented, so is $\rho (\Gamma)$. As $\rho
(\Gamma)$ is also linear, it follows from Malcev's theorem, \cite{Ma}, that it is
residually finite. Further by Selberg's lemma \cite{Se} (see also \cite{Pl}) the image
$\rho (\Gamma)$ is virtually torsion-free.

In other words, there exists a finite index subgroup
$\Gamma_1$ of $\Gamma$ such that $\rho (\Gamma_1)$ is torsion-free. Since $\rho (\Z)$
is finite by Lemma \ref{fin}, it follows that 
$\rho (\Z \cap \Gamma_1)$ is trivial.
\end{proof}

\section{Sasakian groups}

We refer to \cite{BG} for definition and basic properties of Sasakian manifolds. 
Fundamental groups of closed Sasakian manifolds will be called {\it Sasakian groups}.
The Sasakian structure on any Sasakian manifold $M$ can be suitably perturbed producing
a quasi-regular Sasakian structure \cite{Ru}, \cite[p.~161, Theorem~1.2]{OV}. In
particular, every Sasakian group is the fundamental group of some closed quasi-regular
Sasakian manifold. In view of this, henceforth all Sasakian manifolds considered here
will be assumed to be quasi-regular.

Let $M$ be a quasi-regular closed Sasakian manifold.
The K\"ahler orbifold base $M/{\rm U}(1)$ of $M$ will be denoted by $B$ \cite{Ru},
\cite[p.~208, Theorem~7.1.3]{BG}. Let
\begin{equation}\label{f}
f\, :\, M\, \longrightarrow\, B
\end{equation}
be the quotient map. The K\"ahler form on $B$ will be denoted by $\omega$.
Fix a point $x_0\, \in\, M$ such that the fiber $$F\,:=\,
f^{-1}(f(x_0))$$ of $f$ over $f(x_0)$ is regular, meaning the action of ${\rm U}(1)$
on $F$ is free. We denote
$$
\Gamma\, :=\, \pi_1(M,\, x_0)\, .
$$

\begin{proposition}\label{finsas}
Let $G$ be a reductive complex affine algebraic
group, and let $$\rho\,:\, \Ga
\,\longrightarrow\, G$$ be a homomorphism whose image is Zariski dense in $G$.
Then $\rho (\pi_1(F,\, x_0))$ is a finite subgroup of $G$.
\end{proposition}

\begin{proof}
For the map $f$ in \eqref{f}, the Sasakian manifold $M$ is a principal ${\rm U}(1)$--bundle
over the orbifold $B$. We have the homotopy exact sequence for this fibration:
$$\pi_2^{orb}(B,\, f(x_0)) \,\longrightarrow\, \Z \,= \,\pi_1(F,\,
x_0) \, \stackrel{\eta}{\longrightarrow}\, \pi_1(M,\, x_0)
\,\longrightarrow\,\pi_1^{orb} (B,\, f(x_0)) \,\longrightarrow \,1\, ,$$
where $\pi_i^{orb} (B,\, f(x_0))$ is the orbifold homotopy group of $B$.

If the image of $\pi_2^{orb}(B,\, f(x_0))$ in $\pi_1(F,\, x_0)$ is non-trivial, say $p\Z$, then we have an
exact sequence 
$$ 1\,\longrightarrow \,\eta(\pi_1(F,\, x_0))\,=\,
\Z/p\Z \,\longrightarrow\, \pi_1(M,\, x_0) \,\longrightarrow\,
\pi_1^{orb} (B,\, f(x_0)) \,\longrightarrow\, 1\, . $$
The lemma follows in this case because $\rho (\pi_1(F,\, x_0))$ is a quotient
of $\Z/p\Z$.

If the image of $\pi_2^{orb}(B,\, f(x_0))$ in $\pi_1(F,\, x_0)$ is trivial, we have a
short exact sequence 
\begin{equation}\label{extn}
1\,\longrightarrow\, \Z \,\longrightarrow\, \pi_1(M,\, x_0) \,\longrightarrow\,
\pi_1^{orb} (B,\, f(x_0)) \,\longrightarrow\, 1\, ,
\end{equation}
and hence $\pi_1(M,\, x_0)$ is a central extension of $\pi_1^{orb}(B)$ by $ \Z$. 

In view of Lemma \ref{fin}, it suffices to show that the Euler class of the extension 
in \eqref{extn} is non-zero. This is equivalent to showing that the first Chern class 
of the principal ${\rm U}(1)$--bundle $M$ over $B$ in \eqref{f} is non-zero.

Consider the connection $\nabla$ on the principal
${\rm U}(1)$--bundle $M\, \stackrel{f}{\longrightarrow}\, B$ given by the Sasakian
metric on $M$. So the horizontal distribution for $\nabla$ is given by the
orthogonal complement of the Reeb vector field $\xi$ on $M$ associated to the
action of ${\rm U}(1)$ on it. The curvature of $\nabla$ is the K\"ahler form
$\omega$ on $B$. The cohomology class of this curvature is the first Chern class
of the principal ${\rm U}(1)$--bundle $M$ over $B$. The cohomology class
$[\omega]\,\in\, H^2(B,\, {\mathbb Q})$ of $\omega$
is non-zero because $\omega$ is a K\"ahler form. Hence the Euler class of the extension in
\eqref{extn} is non-zero. As noted before, this completes the proof using Lemma \ref{fin}.
\end{proof}

\begin{corollary}\label{finsascor}
Let $\rho\,:\, \Ga \,\longrightarrow\, G$ be a Zariski dense representation, where $G$ as before
is a complex connected reductive affine algebraic group.
Then $M$ admits a finite-sheeted unramified cover $M_1$ such that for any fiber $F_1$
(not necessarily smooth) of the induced quasi-regular Sasakian structure on $M_1$, the
image $\rho (\pi_1(F_1))$ is trivial. (The image of $\pi_1(F_1)$ in $\pi_1(M,\, x_0)$
is unique up to a conjugation --- it depends on the choice of a base point in $F_1$ and a homotopy class of path from $x_0$
to the image of the base point in $M$.)

The same conclusion holds for a representation $\rho\,:\, \Ga \,\longrightarrow\, G$ such
that the Zariski closure of $\rho(\Gamma)$ in $G$ is reductive.
\end{corollary}

\begin{proof}
Take $M_1$ to be the \'etale Galois covering of $M$ for the finite index subgroup
$\Gamma_1\, \subset\, \Gamma$ in the proof of Corollary \ref{torsionfree}. So
$\rho(\pi_1(M_1))$ is torsion-free. The image $\rho (\pi_1(F,\, x_0))$ is finite by
Proposition \ref{finsas}, and hence $\rho (\pi_1(F_1))$ is also a finite group. Therefore,
$\rho (\pi_1(F_1))$ is trivial because it is also torsion-free.

If the Zariski closure of $\rho(\Gamma)$ in $G$ is reductive, then replace $G$ by the
connected component of the Zariski closure of $\rho(\Gamma)$ in $G$ containing the
identity element, and also replace $M$ by the \'etale Galois covering $M'$ of $M$ such that
$\rho(\pi_1(M'))$ is contained in the above connected component. Now the second
part follows from the first part.
\end{proof}

Let $\rho\, :\, \Gamma\, \longrightarrow\, G$ be a homomorphism with $G$ as above. The Zariski closure 
of $\rho(\Gamma)$ in $G$ will be denoted by $\overline{\rho(\Gamma)}$. The connected 
component of $\overline{\rho(\Gamma)}$ containing the identity element will be
denoted by $\overline{\rho(\Gamma)}_0$.

\begin{lemma}\label{lem1}
If $\overline{\rho(\Gamma)}_0$ is not reductive, then $\overline{\rho(\Gamma)}_0$ is
contained in some proper parabolic subgroup of $G$.
\end{lemma}

\begin{proof}
Since $\overline{\rho(\Gamma)}_0$ is not reductive, the unipotent radical of
it is nontrivial. Denote the unipotent radical $R_u(\overline{\rho(\Gamma)}_0)$
by $S_0$. The normalizer of $S_0$ in $G$ will be denoted by $N_1$. The
unipotent radical of $N_1$ will be denoted by $S_1$. Inductively, define $S_i$
to be the unipotent radical of $N_i$, and $N_{j+1}$ to be the normalizer of
$S_j$ in $G$. We have
$$
\cdots\, \subset\,S_i\, \subset\, S_{i+1} \, \subset\, \cdots\, \subset\,
N_{j+1}\, \subset\, N_{j} \, \subset\, \cdots\,
$$
and $\bigcup_i S_i$ is the unipotent radical of $\bigcap_j N_j$, while
$\bigcap_j N_j$ is the normalizer of $\bigcup_i S_i$ in $G$. Therefore,
$\bigcap_j N_j$ is a parabolic subgroup of $G$; see
\cite[p.~185, \S~30.3]{Hu} for more details. Note that $\bigcap_j N_j$ is a
proper subgroup of $G$ because its unipotent radical $\bigcup_i S_i$ is nontrivial
as it contains $S_0$.
\end{proof}

We now give a criterion for $\overline{\rho(\Gamma)}$ to be reductive under the assumption
that $G\, =\, \text{GL}(r,{\mathbb C})$.
For a homomorphism $$\rho\, :\, \Gamma\, \longrightarrow\, \text{GL}(r,{\mathbb C})$$
the Zariski closure $\overline{\rho(\Gamma)}$ is reductive if and only if
${\mathbb C}^r$ decomposes into a direct sum of irreducible
$\rho(\Gamma)$--modules. Indeed, if $\overline{\rho(\Gamma)}$ is reductive, then the
$\overline{\rho(\Gamma)}$--module ${\mathbb C}^r$ decomposes into a direct sum of irreducible
$\overline{\rho(\Gamma)}$--modules. Hence ${\mathbb C}^r$ decomposes into a direct sum of irreducible
$\rho(\Gamma)$--modules. Conversely, if ${\mathbb C}^r$ decomposes into a direct sum of irreducible
$\rho(\Gamma)$--modules, then from Lemma \ref{lem1} it follows that $\overline{\rho(\Gamma)}$
is reductive.

\section{Sasakian Higgs bundles}

\subsection{Partial connection}

Take a connected $C^\infty$ manifold $X$. Take any $C^\infty$ subbundle of positive rank
$$
S\, \subset\, TX\otimes_{\mathbb R} {\mathbb C}
$$
which is closed under the
operation of Lie bracket of vector fields; such a subbundle is called integrable.
We have the dual of the inclusion map of $S$ in $TX\otimes_{\mathbb R} {\mathbb C}$
\begin{equation}\label{S1}
q_S\, :\, T^*X\otimes_{\mathbb R} {\mathbb C}\, =\, (TX
\otimes_{\mathbb R} {\mathbb C})^*\, \longrightarrow\, S^*\, .
\end{equation}

A \textit{partial connection on $E$ in the direction of} $S$ is a $C^\infty$
differential operator
$$
D\, :\, E \, \longrightarrow\, S^*\otimes E
$$
satisfying the Leibniz condition, which says that 
$$
D(fs) \,=\, fD(s) + q_S(df)\otimes s
$$
for a smooth section $s$ of $E$ and a smooth function $f$ on $X$,
where $q_S$ is the projection in \eqref{S1}.

Since the distribution $S$ is integrable, the smooth sections of ideal subbundle of the exterior
algebra bundle $\bigwedge T(^*X\otimes_{\mathbb R} {\mathbb C})$ generated by
$\text{kernel}(q_S)$ is 
closed under the exterior derivation. Therefore, we have an induced exterior
derivation acting on the smooth sections of $S^*$
\begin{equation}\label{ed}
\widehat{d}\, :\, S^*\, \longrightarrow\, \bigwedge\nolimits^2 S^*
\end{equation}
which is a differential operator of order one.

Let $D$ be a partial connection on $E$ in the direction of $S$. Consider the
differential operator
$$
D_1\, :\, S^*\otimes E\, \longrightarrow\, (\bigwedge\nolimits^2 S^*)
\otimes E
$$
defined by
$$
D_1(\theta\otimes s) \, =\, \widehat{d}(\theta)\otimes s
-\theta\wedge D(s)\, ,
$$
where $\widehat{d}$ is constructed in \eqref{ed}. The composition
\begin{equation}\label{ed2}
E \, \stackrel{D}{\longrightarrow}\, S^*\otimes E
\, \stackrel{D_1}{\longrightarrow}\, (\bigwedge\nolimits^2 S^*)
\otimes E
\end{equation}
is $C^\infty(X)$--linear. Therefore, the composition in \eqref{ed2} defines a
$C^\infty$ section
\begin{equation}\label{ed3}
{\mathcal K}(D)\, =\,
C^\infty(X, \,(\bigwedge\nolimits^2 S^*) \otimes E\otimes E^*)\, =\,
C^\infty(X, \,(\bigwedge\nolimits^2 S^*) \otimes \text{End}(E))\, .
\end{equation}

The section ${\mathcal K}(D)$ in \eqref{ed3} is called the \textit{curvature}
of $D$. If
$$
{\mathcal K}(D)\, =\, 0\, ,
$$
then the partial connection $D$ is called \textit{flat}.

\subsection{Holomorphic hermitian vector bundles}\label{se:holvec}

Let $M$ be a compact quasi-regular Sasakian manifold. The Riemannian metric
and the Reeb vector field on $M$ will be denoted by $g$ and $\xi$
respectively.
The almost complex structure on the orthogonal complement $$\xi^\perp\, \subset\, TM$$
for $g$ produces a type decomposition
$$
\xi^\perp\otimes_{\mathbb R} {\mathbb C}\,=\, F^{1,0}\oplus F^{0,1}\, .
$$
Define $F^{p,q}\, :=\, (\bigwedge^p F^{1,0})\otimes (\bigwedge^q F^{0,1})$. Let
\begin{equation}\label{w01}
\widetilde{F}^{0,1}\, :=\, F^{0,1}\oplus (\xi\otimes_{\mathbb R}
\mathbb C)\, \subset\, TM\otimes_{\mathbb R} {\mathbb C}
\end{equation}
be the distribution. It is known that
this distribution $\widetilde{F}^{0,1}$ is integrable \cite[p.~550, Lemma~3.2]{BS}.

A {\em Sasakian complex vector bundle} on the Sasakian manifold
$(M,\, g,\, 
\xi)$ is a pair $(E,\, D_0)$, where $E$ is a $C^\infty$ complex vector bundle
on $M$, and $D_0$ is a partial connection on $E$ in the direction $\xi$.

A {\em hermitian structure} on a Sasakian complex vector bundle $(E,\, D_0)$ is a 
$C^\infty$ hermitian structure on the complex vector bundle $E$ preserved by the 
partial connection $D_0$.

\begin{definition}\label{def1}
A {\em holomorphic structure} on a Sasakian complex vector bundle $(E,\, D_0)$
is a flat partial connection $D$ on $E$ in the direction of
$\widetilde{F}^{0,1}$ (constructed in \eqref{w01}) satisfying the
compatibility condition that $D_0$ coincides with the partial
connection on $E$, in the direction of $\xi$, defined by $D$.

A {\em Sasakian holomorphic vector bundle} is a Sasakian complex vector
bundle equipped with a holomorphic structure.
\end{definition}

Let $(E,\, D)$ be a Sasakian holomorphic
vector bundle on $M$ equipped with a hermitian metric $h$.
There is a unique connection $\nabla$ on the
complex vector bundle $E$ satisfying the following two conditions:
\begin{enumerate}
\item $\nabla$ preserves $h$, and

\item the partial connection on $E$ in the direction of
$\widetilde{F}^{0,1}$ induced by $\nabla$ coincides with $D$.
\end{enumerate}
The second condition is equivalent to the condition that the curvature
of $\nabla$ is a smooth section of $(F^{1,1})^*\otimes E \otimes E^*$.

Let ${\mathcal K}(E,h)\, :=\, {\mathcal K}(\nabla)$ be the curvature of the above connection 
$\nabla$; as mentioned above, it is a section of $(F^{1,1})^*\otimes E \otimes E^*$.

Let $(E,\, D_{E})$ and $(E',\, D_{E'})$ be two Sasakian holomorphic vector bundles on
$(M,\, g,\, \xi)$. A fiberwise $\mathbb C$--linear $C^\infty$ map
$$
\Psi \,: \,E'\,\longrightarrow\, E''
$$
is called {\em holomorphic}, if $\Psi$ intertwines $D_E$ and $D_{E'}$. A holomorphic
section of $(E',\, D_{E'})$ is a holomorphic homomorphism to it from
the trivial complex line bundle on $M$ equipped with the trivial Sasakian holomorphic structure
given by the trivial partial connection on it in the direction of $\widetilde{F}^{0,1}$.

Using the Levi--Civita connection on $M$ for $g$, the vector bundle $(F^{1,0})^*$ gets
a flat partial connection along $\widetilde{F}^{0,1}$, thus making $(F^{1,0})^*$ a
Sasakian holomorphic vector bundle. Let $E$ be a Sasakian 
holomorphic vector bundle on $(M,\, g,\, \xi)$. Then $\text{End}(E)$ also has the 
structure of a Sasakian holomorphic vector bundle. Hence
$\text{End}(E)\otimes (F^{1,0})^*$ is a Sasakian holomorphic vector bundle.

A Higgs field on a Sasakian holomorphic vector bundle $(E,\, D_{E})$ is a holomorphic section $\theta$ of
$\text{End}(E)\otimes (F^{1,0})^*$ such that the section $\theta\wedge\theta$ of
$\text{End}(E)\otimes (F^{2,0})^*$ vanishes identically. A \textit{Higgs vector
bundle} on the Sasakian manifold $M$ is a Sasakian holomorphic vector bundle on $M$ equipped
with a Higgs field.

Let $H$ be a Lie group and $q\, :\, E_H\,\longrightarrow\, M$ a $C^\infty$ principal
$H$--bundle on $M$. Let $$dq\, :\, TE_H\,\longrightarrow\, q^* TM$$ be the differential
of $q$. A partial 
connection on $E_H$ in the direction of $\xi$ is a $H$--equivariant homomorphism
$$
D_0\, :\, q^*\xi \,\longrightarrow\, q^* TE_H
$$
such that $(dq)\circ D_0$ coincides with the identity map of $q^*\xi$. In other words,
$D_0$ is a $H$--equivariant lift of $\xi$ to 
the total space of $E_H$. A \textit{Sasakian principal $H$--bundle} on $M$ is a
$C^\infty$ principal 
$H$--bundle $E_H$ on $M$ equipped with a partial connection in the direction of 
$\xi$. Let $H$ be a complex Lie group and $(E_H,\, D_0)$ a Sasakian principal 
$H$--bundle on $M$. A holomorphic structure on $E_H$ is an $H$--invariant lift $D$ of the 
subbundle $\widetilde{F}^{0,1}\, \subset\, TM\otimes_{\mathbb R} {\mathbb C}$
(see \eqref{w01}) to $TE_H\otimes_{\mathbb R} {\mathbb C}$ such that restriction of
$D$ to $\xi$ coincides with
the given lift $D_0$. A \textit{holomorphic Sasakian principal $H$--bundle} is
a Sasakian principal $H$--bundle equipped with a holomorphic structure.

Let $(E_H,\, D)$ be a holomorphic Sasakian principal $H$--bundle on $M$, and let $H_1$ be a complex
Lie subgroup of $H$. A $C^\infty$ reduction of structure group $E_{H_1}\, \subset\, E_H$
to $H_1$ is called \textit{holomorphic} if for every $z\, \in\, E_{H_1}$,
the image of $\widetilde{F}^{0,1}$ in $T_zE_H\otimes
{\mathbb C}$ under $D$ is contained in $T_zE_{H_1}\otimes {\mathbb C}$.

Let $G$ be a connected complex reductive affine algebraic group. Fix a maximal compact 
subgroup $$K_G\, \subset\, G\, .$$ Let $(E_G,\, D)$ be a Sasakian holomorphic 
principal $G$--bundle. A Hermitian structure on it is a $C^\infty$ reduction of 
structure group
$$
E_G\, \supset\, E_K\, \stackrel{q}{\longrightarrow}\, M
$$
such that for any $z\, \in\, E_K$, the image $D(\xi(q(z)))\,\in\, T_zE_G$ is contained
in $T_z E_K$. Note that this condition is equivalent to the condition that $D$ induces
a Sasakian principal $K$--bundle structure on $E_K$.

For any Hermitian structure $E_K$ as above, there is a unique connection 
$\nabla$ on the principal $K$--bundle $E_K$ such that for every $z\in\, E_K$, the lift of
$\widetilde{F}^{0,1}(z)$ to $T_z E_K\otimes_{\mathbb R} {\mathbb C}$ given by $\nabla$ coincides
with the one given by $D$. This condition is equivalent to to the condition that the
curvature of $\nabla$ is a section of $(F^{1,1})^*\otimes \text{ad}(E_G)$, where $${\rm ad}(E_G)\,=\,
E_G\times^G {\mathfrak g}$$ is the vector bundle on $M$ associated to $E_G$ for the
adjoint action of $G$ on its Lie algebra $\mathfrak g$; this $\text{ad}(E_G)$ is also called
the adjoint vector bundle for $E_G$.

Let $(E_G,\, D_0)$ be a Sasakian principal $G$--bundle. Note that any vector bundle associated to 
$E_G$ has the structure of a Sasakian vector bundle using $D_0$. In particular, the adjoint vector 
bundle $\text{ad}(E_G)$ is a Sasakian vector bundle. A holomorphic structure $D$ on $(E_G,\, D_0)$ 
produces a holomorphic structure on any associated vector bundle.

A Higgs field on a Sasakian holomorphic
principal $G$--bundle $(E_G,\, D)$ is a holomorphic section of
$\text{ad}(E_G)\otimes (F^{1,0})^*$ such that the section $\theta\wedge\theta$ of
$\text{ad}(E)\otimes (\bigwedge^2 F^{1,0})^*$ vanishes identically.

Let $((E_G,\, D),\, \theta)$ be a Sasakian $G$--Higgs bundle, and let $P\,\subsetneq\, G$ be a maximal
parabolic subgroup. A reduction $E_P$ of $((E_G,\, D),\, \theta)$ to $P$ over a big open subset
$U$ means the following:
\begin{itemize}
\item $U\, \subset\, M$ is an open subset preserved by the action of ${\rm U}(1)$ such that
the complement $B\setminus f(U)$ (the map $f$ is defined in \eqref{f}) is a complex analytic
subset of complex codimension at least two,

\item $E_P\, \subset\, E_G\vert_U$ is a holomorphic reduction of structure group over $U$, and

\item the Higgs field $\theta$ is a section of the subbundle $\text{ad}(E_P)\otimes (F^{1,0})^*\,
\subset\, \text{ad}(E_G)\otimes (F^{1,0})^*$.
\end{itemize}

A Sasakian $G$--Higgs bundle $((E_G,\, D),\, \theta)$ is called
\textit{stable} (respectively, \textit{semistable}) if for all maximal
parabolic subgroup $P\,\subsetneq\, G$ and every reduction $E_P$ of $((E_G,\, D),\, \theta)$ to $P$
over every big open subset $U$,
$$
\text{degree}({\rm ad}(E_P))\, <\, 0 \ \ \ \text{(respectively, }~\text{degree}({\rm ad}(E_P))\,
\leq\, 0 {\rm )}\, .
$$

A Sasakian $G$--Higgs bundle $((E_G,\, D),\, \theta)$ is called \textit{polystable} if
the following conditions hold:
\begin{enumerate}
\item $((E_G,\, D),\, \theta)$ is semistable, and

\item there is a Levi subgroup $L$ of some parabolic subgroup of $G$ and a holomorphic reduction
$E_L\, \subset\, E_G$ to $L$, such that
\begin{itemize}
\item $\theta$ is a section of the subbundle $\text{ad}(E_L)\otimes (F^{1,0})^*\,
\subset\, \text{ad}(E_G)\otimes (F^{1,0})^*$, and

\item $((E_L,\, D),\, \theta)$ is stable.
\end{itemize}
\end{enumerate}

\begin{definition}\label{def2}
A Sasakian $G$--Higgs bundle $(V,\, \theta)$ on $M$ will be called {\it basic} if 
$(V,\, \theta)$ descends to the projective variety $M/\text{U}(1)$. In other words,
there is a Higgs $G$--Higgs bundle $(W,\, \theta')$ on $M/\text{U}(1)$ such that
$(f^*W,\, f^*\theta')$ equipped with the natural Sasakian $G$--Higgs bundle structure
is isomorphic $(V,\, \theta)$.

The characteristic classes of a basic Sasakian $G$--Higgs bundle $(V,\, \theta)$
are defined to be the corresponding characteristic classes of the above
principal $G$--bundle $W$.

A Sasakian $G$--Higgs bundle $(V,\, \theta)$ on $M$ will be called {\it virtually basic}
if there is finite \'etale Galois covering $\widetilde{M}\, \longrightarrow\, M$
such that the pullback of $(V,\, \theta)$ to $\widetilde{M}$ is basic. Note that
$\widetilde{M}$ is also a quasi-regular Sasakian manifold.

The characteristic classes of a virtually basic Sasakian $G$--Higgs bundle $(V,\, \theta)$
on $M$ are defined to be the corresponding characteristic classes of the
basic Sasakian $G$--Higgs bundle on $\widetilde{M}$.
\end{definition}

Note that the condition that the rational characteristic classes of a virtually basic 
Sasakian $G$--Higgs bundle on $M$ vanish does not depend on the choice of the 
covering $\widetilde{M}$. This condition can be interpreted in terms of equivariant
cohomology.

An example in Section \ref{ex} shows why the above class of Higgs bundles is relevant.

\section{The Flat Connection-Higgs Bundle Correspondence}

The purpose of this Section is to establish the following main Theorem of this paper:

\begin{theorem} \label{main}
Let $G$ be a connected reductive complex affine algebraic group. Let $M$ be a
quasi-regular Sasakian manifold
with fundamental group $\Gamma$. Any homomorphism
$$\rho\, :\, \Gamma\,\longrightarrow\, G$$ with the Zariski closure of $\rho(\Gamma)$
reductive canonically gives a virtually basic polystable
principal $G$--Higgs bundle on $M$ with
vanishing rational characteristic classes. Conversely, any 
virtually basic polystable principal $G$--Higgs bundle on $M$ with 
vanishing rational characteristic classes corresponds to a flat principal
$G$--bundle on $M$ with the property that the Zariski closure of the monodromy
representation is reductive.
\end{theorem}

Section \ref{rep2higgs} proves the forward direction and Section \ref{higgs2rep} 
establishes the converse direction of the above Theorem.

\subsection{From flat connections to Higgs bundles}\label{rep2higgs}

As before, $G$ is a connected reductive complex affine algebraic group.
Let $$\rho\, :\, \Gamma\,\longrightarrow\, G$$ be a
homomorphism such that the Zariski closure $\overline{\rho(\Gamma)}$ is reductive.

Let
\begin{equation}\label{varp}
\varphi\, :\, \widetilde{M}\, \longrightarrow\, M
\end{equation}
be a finite \'etale Galois covering. Then $\widetilde{M}$ is also a quasi-regular Sasakian
manifold. The Reeb vector field on $\widetilde{M}$, which is simply the inverse image
of $\xi$, will be denoted by $\widetilde{\xi}$; note that the differential
$$
d\varphi\, :\, T\widetilde{M}\, \longrightarrow\, \varphi^*TM
$$
is an isomorphism. Fix a point $\widetilde{x}_0\,\in\,
\widetilde{M}$ over $x_0$. Each orbit of $\widetilde{\xi}$ defines
a conjugacy class in $\pi_1(\widetilde{M},\, \widetilde{x}_0)$; however, all regular orbits
define the same conjugacy class. The normal
subgroup of $\pi_1(\widetilde{M},\, \widetilde{x}_0)$ generated by all the conjugacy
classes given by the orbits of $\widetilde{\xi}$ will be denoted by $\Gamma_0$.

{}From Corollary \ref{finsascor} we know that there is a finite \'etale Galois covering
$\varphi$ as above such that the composition
$$
\Gamma_0\, \hookrightarrow\, \pi_1(\widetilde{M},\, \widetilde{x}_0)
\, \stackrel{\varphi_*}{\longrightarrow}\, \Gamma \,
\stackrel{\rho}{\longrightarrow}\, G
$$
is the trivial homomorphism. Fix such a covering $\varphi$.

Let
$$
\rho'\, :\, \pi_1(\widetilde{M},\, \widetilde{x}_0)\, \longrightarrow\, G
$$
be the composition $\pi_1(\widetilde{M},\, \widetilde{x}_0)
\, \stackrel{\varphi_*}{\longrightarrow}\, \Gamma \,
\stackrel{\rho}{\longrightarrow}\, G$. For notational convenience, the Galois group
$\text{Gal}(\varphi)$ for the covering $\varphi$ will henceforth be denoted by $\Pi$.

Since $\rho'\vert_{\Gamma_0}$ is the trivial homomorphism, the homomorphism $\rho'$
descends to a homomorphism, to $G$, from the fundamental group of the quotient space
$\widetilde{B}\, :=\, \widetilde{M}/\text{U}(1)$.

It should be clarified that here we consider $\widetilde{B}$ just as a quotient space 
and not as an orbifold. Note that the necessary and sufficient condition for the 
homomorphism $\rho'$ to descend to a homomorphism to $G$ from the fundamental group of the 
orbifold $\widetilde{M}/\text{U}(1)$ is that the restriction of
$\rho'$ to the normal subgroup of $\pi_1(\widetilde{M},\, \widetilde{x}_0)$ generated
by the regular orbits of $\widetilde{\xi}$ is trivial. The stronger condition
that $\rho'\vert_{\Gamma_0}$ is trivial ensures that $\rho'$
descends to a homomorphism, to $G$, from the fundamental group of the quotient space
$\widetilde{B}$.

The image of $\widetilde{x}_0$
in $\widetilde{B}$ will be denoted by $y_0$. Let
$$
\rho'_0\, :\, \pi_1(\widetilde{B},\, y_0)\, \longrightarrow\, G
$$
be the homomorphism given by $\rho'$. Let
\begin{equation}\label{varp0}
\beta\, :\, (E'_G,\, \nabla)\, \longrightarrow\, \widetilde{B}
\end{equation}
be the associated flat principal $G$--bundle.

The map $\varphi$ in \eqref{varp} descends to a map
\begin{equation}\label{varp2}
\widetilde{\varphi}\, :\, \widetilde{M}/\text{U}(1)\,=\,
\widetilde{B} \, \longrightarrow\, M/\text{U}(1)\,=\, B\, .
\end{equation}
The action of the Galois group $\Pi$ on $\widetilde{M}$ descends to
$\widetilde{B}$. Consequently, $\widetilde{\varphi}$ in \eqref{varp2}
is an \'etale Galois covering with Galois group $\Pi$.

Let
\begin{equation}\label{wtf}
\widetilde{f}\, :\, \widetilde{M}\, \longrightarrow\,\widetilde{B}
\end{equation}
be the quotient map. Since the flat principal $G$--bundle
$$
(\widetilde{f}^* E'_G,\, \widetilde{f}^*\nabla)\, \longrightarrow\, \widetilde{M}
$$
(see \eqref{varp0}) is the pull back of the flat principal $G$--bundle on $M$ associated to $\rho$. 
Therefore, $(\widetilde{f}^* E'_G,\, \widetilde{f}^*\nabla)$ is $\Pi$--equivariant, meaning it is 
equipped with a lift of the Galois action of $\Pi$ on $\widetilde M$. This action of $\Pi$ on
$(\widetilde{f}^* E'_G,\, \widetilde{f}^*\nabla)$ descends to an action of $\Pi$ on
$(E'_G,\, \nabla)$, because the map $\widetilde{f}$ is $\Pi$--equivariant.
Consequently, $(E'_G,\, \nabla)$ is a $\Pi$--equivariant flat principal $G$--bundle.

Take a desingularization
$$
\delta\, :\, Z\, \longrightarrow\, \widetilde{B}\, .
$$
Let $(V,\, \theta)$ be the polystable $G$--Higgs bundle on $Z$ corresponding to
the flat principal $G$--bundle $(\delta^*E'_G,\, \delta^*\nabla)$
on $Z$ \cite{Co}, \cite{Do}, \cite{Si2}, \cite{BSu}. All the rational characteristic classes
of $V$ vanish because the principal $G$--bundle $V$ is topologically isomorphic
to $\delta^*E'_G$. The
restriction of $(\delta^*E'_G,\, \delta^*\nabla)$ to any fiber of $\delta$ is trivial.
Hence $(V,\, \theta)$ descends to $\widetilde{B}$; this descended
$G$--Higgs bundle on $\widetilde{B}$ will be denoted by $(V_0,\, \theta_0)$ (see
\cite{ES}). The action of the Galois group $\Pi$ on $(E'_G,\, \nabla)$ produces
an action of $\Pi$ on $(V_0,\, \theta_0)$.

Therefore, the pullback $(\widetilde{f}^*V_0,\, \widetilde{f}^*\theta_0)$ by the
map $\widetilde f$ in \eqref{wtf} is a $\Pi$--equivariant principal $G$--Higgs bundle.
Hence, $(\widetilde{f}^*V_0,\, \widetilde{f}^*\theta_0)$ descends to $M$.%

Note that the above constructed $G$--Higgs bundle on $M$ is virtually basic; see 
Definition \ref{def2}.%
All the rational characteristic classes of this virtually basic $G$--Higgs bundle
on $M$ vanish because all the rational characteristic classes of $V$ vanish
(see Definition \ref{def2}).

\subsection{From Higgs bundles to flat connections}\label{higgs2rep}

Let $(V,\, \theta)$ be a Sasakian virtually basic $G$--Higgs bundle on $M$
such that all the rational characteristic classes of $V$
vanish. Fix a Galois \'etale covering
$$
\varphi\, :\, \widetilde{M}\, \longrightarrow\, M
$$
such that $(\varphi^*V,\, \varphi^*\theta)$ is basic. Let
$(V',\, \theta')$ be the $G$--Higgs bundle on $\widetilde{B}\,:=\, \widetilde{M}/\text{U}(1)$.
As before, the Galois group $\text{Gal}(\varphi)$ will be denoted by $\Pi$.

Fix a desingularization
$$
\delta\, :\, Z\, \longrightarrow\,\widetilde{B}\, .
$$
The Higgs $G$--bundle $(\delta^*V',\, \delta^*\theta')$ corresponds to a flat principal
$G$--bundle $(F_G,\, \nabla)$ on $Z$ \cite{Si1}, \cite{Si2}, \cite{Hi}, \cite{BSu}.
The Zariski closure of the monodromy representation for $\nabla$ is reductive. Since
$$
\delta_*\,:\, \pi_1(Z)\, \longrightarrow\, \pi_1(\widetilde{B})
$$
is an isomorphism \cite[p.~203, Theorem~(7.5.2)]{Ko}, the above flat principal
$G$--bundle $(F_G,\, \nabla)$ descends to a flat principal
$G$--bundle $(F'_G,\, \nabla')$ on $\widetilde{B}$; see \cite{ES}.

The above flat principal $G$--bundle $(F'_G,\, \nabla')$ on $\widetilde{B}$ is
$\Pi$--equivariant, because $(V',\, \theta')$ is $\Pi$--equivariant. Hence the pullback
$(\widetilde{f}^*F'_G,\, \widetilde{f}^*\nabla')$ is also $\Pi$--equivariant,
where $\widetilde f$, as before, is the quotient map $\widetilde{M}\,\longrightarrow\,
\widetilde{M}/\text{U}(1)$. Consequently,
this flat principal $G$--bundle $(\widetilde{f}^*F'_G,\, \widetilde{f}^*\nabla')$
descends to a flat principal $G$--bundle on $M$. The monodromy
representation for this flat principal $G$--bundle on $M$ is reductive because
the monodromy representation for $\nabla$ is reductive.

\subsection{An example}\label{ex}

We now give an example which shows that the condition ``virtually basic'' in
Definition \ref{def2} is essential.

Consider the unit three-sphere $$M\, =\, S^3\, :=\, \{(z_1,\, z_2)\, \in\, {\mathbb C}^2\,
\mid\, |z_1|^2+|z_2|^2\,=\, 1\}$$ equipped with the standard Sasakian structure. The
Riemannian metric is the one obtained by restricting the standard metric on
${\mathbb C}^2$. The Reeb vector field is given by the diagonal action of
${\rm U}(1)$ on ${\mathbb C}^2$ for the standard action of it on ${\mathbb C}$. So
the action of any $c\, \in\, {\rm U}(1)$ sends any $(z_1,\, z_2)\, \in\, {\mathbb C}^2$
to $(c\cdot z_1,\, c\cdot z_2)$.
So in \eqref{f}, the surface $B$ is ${\mathbb C}{\mathbb P}^1$, and $f$
defines the Hopf fibration on it.

Take any integer $n\, \not=\, 0$, and consider the nontrivial line bundle
$L\, :=\, {\mathcal O}_{{\mathbb C}{\mathbb P}^1}(n)$ on ${\mathbb C}{\mathbb P}^1$
of degree $n$. The pullback $f^*L$ has a tautological action of ${\rm U}(1)$
simply because it is pulled back from ${\mathbb C}{\mathbb P}^1$. Let $\mathbb L$
denote this holomorphic line bundle on the Sasakian manifold $M$. Let ${\mathbb L}_0$
denote the trivial holomorphic line bundle
on $M$ equipped with the trivial ${\rm U}(1)$--action. We note that ${\mathbb L}$
is not holomorphically isomorphic to ${\mathbb L}_0$ because ${\mathbb L}$ descends
to $L$ on ${\mathbb C}{\mathbb P}^1$, while ${\mathbb L}_0$ descends to the
trivial holomorphic line bundle on ${\mathbb C}{\mathbb P}^1$. Also note that the
complex line bundle $f^*L$ is topologically trivial because $H^2(M,\, {\mathbb Z})
\,=\, 0$. Therefore, $({\mathbb L},\, 0)$ is a nontrivial Higgs line bundle on
$M$ of vanishing characteristic classes; being of rank one $({\mathbb L},\, 0)$ is
polystable.

On the other hand, there is no nontrivial flat line bundle
on $S^3$ as it is simply connected. So Theorem \ref{main} is not valid for
$S^3$ if we drop the condition ``virtually basic'' in the statement.

\section*{Acknowledgements}

We thank Carlos Simpson for a useful discussion. The authors acknowledge the support 
of their respective J. C. Bose Fellowships.

\end{document}